\documentclass[12pt, letterpaper]{article}
\usepackage[utf8]{inputenc}

\usepackage{graphicx}
\usepackage[T1]{fontenc}
\usepackage[top=3cm,bottom=3.5cm,left=4cm,right=3.8cm,dvips]{geometry}

 \usepackage[]{algorithm2e}
\usepackage{array,multirow}

\usepackage{enumitem}
\usepackage{url}

\usepackage{amssymb}
\usepackage{epsfig}

\newtheorem{proposition}{Proposition}
\newtheorem{proof}{Proof}

\usepackage{graphicx}
\usepackage{amssymb}
\usepackage{amsmath}

\begin{document}

\title{Devolutionary Genetic Algorithms with Application to the Minimum Labeling Steiner Tree Problem}


\author{Nassim Dehouche\\
Mahidol University International College \\
              Business Administration Division\\
             \small{nassim.deh@mahidol.ac.th}\\ }


\date{}
\maketitle


\begin{abstract}
This paper\footnote{Cite as Dehouche, N., 2018, Devolutionary genetic algorithms with application to the minimum labeling Steiner tree problem. Evolving Systems, 9(2), pp 157–168.\\ 

This is a post-peer-review, pre-copyedit version of an article published in Evolving Systems. The final authenticated version is available online at: \url{https://link.springer.com/article/10.1007/s12530-017-9182-z.}} characterizes and discusses devolutionary genetic algorithms and evaluates their performances in solving the Minimum Labeling Steiner Tree (MLST) problem. We define devolutionary algorithms as the process of reaching a feasible solution by devolving a population of super-optimal unfeasible solutions over time. We claim that distinguishing them from the widely used evolutionary algorithms is relevant. The most important distinction lies in the fact that in the former type of processes, the value function decreases over successive generation of solutions, thus providing a natural stopping condition for the computation process. We show how classical evolutionary concepts, such as crossing, mutation and fitness can be adapted to aim at reaching an optimal or close-to-optimal solution among the first generations of feasible solutions. We additionally introduce a novel integer linear programming formulation of the MLST problem and a valid constraint used for speeding up the devolutionary process. Finally, we conduct an experiment comparing the performances of devolutionary algorithms to those of state of the art approaches used for solving randomly generated instances of the MLST problem. Results of this experiment support the use of devolutionary algorithms for the MLST problem and their development for other NP-hard combinatorial optimization problems.

\end{abstract}

\textbf{Keywords:} Hybrid Meta-heuristics, Integer Linear Programming, Evolutionary computation, Minimum Labelling Steiner tree problem.


\section{Introduction}
\label{S:1}

In biology, the notion that there exist a preferred hierarchy of structure and function among organisms is widely accepted as a fallacy \cite{SA}. Thus, the idea of the devolution of species is considered indiscernible from their evolution. However, in evolutionary computer science, the latter hierarchy is necessary and generally exists in the form of a value function. Evolutionary genetic algorithms constitute a class of meta-heuristic approaches that aim at reaching a close-to-optimal solution to a combinatorial optimization problem, through the improvement of the value function associated with a population of sub-optimal candidate solutions, over time. Good evolutionary genetic algorithms are expected to reach a close-to-optimal solution in a reduced number of generations, therefore in most of these algorithms the value function is improved in each new generation. Moreover, all generations of solutions are feasible and sub-optimal. Therefore, the concern lies in designing a process in which value increases fast enough to reach a satisfying solution in a reduced computation time. However, defining stopping conditions for evolutionary genetic algorithms can be a tedious task \cite{STOP}, such conditions often requiring to be defined on a case by case basis \cite{STOP2}, when they are not arbitrarily defined by the computation time available.\\ 

In numerous combinatorial optimization problems, generating super-optimal unfeasible solutions is a relatively easy task (e.g. through constraints relaxation). We call devolutionary algorithm a computation process in which successive generations of solutions are unfeasible super-optimal solutions and their value typically decreases over time. The goal in such processes is increasing the feasibility of successive generations over time, while trying to limit the decrease in value. Therefore, the main aim in designing devolutionary meta-heuristics lies in reaching an optimal or close-to-optimal solution among the first generations of feasible solutions. \\

In addition to providing natural stopping conditions (i.e. reaching one or a certain number of feasible solutions), there is an intuitive justification to the use of this devolutionary approach for bypassing the issue of premature convergence to local-optima, that is so prevalent in genetic algorithms \cite{local}. Indeed, in evolutionary algorithms, the absolute (i.e. global) value of the properties possessed by successive generations of solution that get passed onto offspring remains unknown all throughout the process. All generated solutions can only be good relative to the sample of solutions generated thus far, whereas with the devolutionary approach, super-optimal solutions in the initial population are expected to possess absolutely good structural properties for a given problem (e.g. a small number of colors for coloring problems), although they are not adequately adapted to this problem. The computation process is oriented in such a way as to pass on these properties to future generations while trying to improve their adaptability to the problem at hand. One can expect withal that generating the initial population of solution would be computationally more demanding in the latter type of processes than in the former, although the design of devolutionary algorithms can feed from advances in linear programming and other methods of generating super-optimal solutions. Another apparent limitation of the devolutionary approach is that the premise of using super-optimal solutions only seems applicable to single-objective combinatorial optimization problems, unlike the evolutionary approach that can be more generally used as a search procedure for a wider variety of tasks, such as multi-objective optimization \cite{111} or machine learning \cite{ML}.\\ 

This paper is an initial attempt to evaluate the pertinence of developing devolutionary genetic algorithms for hard combinatorial optimization problems. We choose to focus on a class of edge coloring problems known as the Minimum Labeling Steiner Tree (MLST) problem, defined in section \ref{problem}. This variant of the Steiner Tree problem presents the advantage of possessing some well-performing exact and heuristic solving methods, a brief review of which can be found in section \ref{review}. However, to the best of our knowledge, linear-programming based approaches have yet to be tested for the MLST. Thus, there exist a theoretical interest in developing such methods and the body of existing methods offers good opportunities for gauging their performances. In this regard, we propose an integer linear programming formulation of the problem in section \ref{formulation}. and describe the proposed devolutionary genetic algorithm in section \ref{proposed}, which also introduces a new class of valid inequalities that we use to solve relaxations of the MLST problem. We compare this algorithm with the aforementioned exact and heuristic methods in section \ref{computation}. Finally, we draw conclusions regarding the results of this experiment and perspective of further development of the proposed approach in section \ref{conclusion}.

\section{Problem statement}
\label{problem}
Given a graph with labeled (or colored) edges, one seeks a spanning tree covering a subset of nodes, known as terminals or basic nodes, whose edges have the least number of distinct labels (or colors). Formally, given $G = (V,E,L)$ a labeled, connected, undirected graph, where $V$ is the set of nodes, $E$ is the set of edges, that are labeled on the set $L$ of labels (or colors). Let $Q$ be the set of nodes that must be connected in a feasible solution. The objective is to find a spanning tree $T$ of the sub-graph connecting all the terminals $Q$ such that the number of colors used by $T$ is minimized. This problem has numerous real-world applications. For example, a multi-modal transportation networks is represented by a graph where each edge is assigned a color, denoting a different company managing that edge, and each node represents a different location. Thus, it is desirable to provide a complete service between a basic set of locations, without cycles, using the minimum number of companies, in order to minimize the costs.

\section{Related work}
\label{review}
The MLST problem is an extension of the well-studied Steiner tree problem \cite{7}, and of the minimum labeling tree spanning problem \cite{8}, which are known to belong to the class of NP-hard combinatorial optimization problems, and for which (evolutionary) genetic algorithms have proved to be successful meta-heuristic solving approaches \cite{9,10}. Just like the two problems it extends, the MLST problem belongs to the class of NP-hard problems. Thus, its most successful solving approach currently known relies on the use of heuristics and meta-heuristics. The problem was first considered in \cite{1} and a heuristic approach known as the Pilot Method, as well as meta-heuristic approaches, namely variable neighbourhood search \cite{2} and Particle Swarm Optimization \cite{6}, were successfully implemented for its resolution. The Pilot Method was found to obtain the best results compared to some meta-heuristic approaches (Tabu Search, Simulated Annealing, and Variable Neighbourhood Search). One can observe that the existing body of work on the MLST problem is exclusively based on graph-theoretic formulations of the problem. To the best of our knowledge, integer linear programming formulations, and their relaxations, have yet to be explored as a possible framework in which to develop meta-heuristic solving approaches to the problem, despite the fact that this research direction, known as Hybrid Meta-heuristics \cite{11} in the literature, has been fruitful for the Steiner tree problem \cite{12}.\\ 

The proposed devolutionary approach falls into the wider category of memetic algorithms, as defined in \cite{121} as \textit{``an evolutionary metaheuristic that can be viewed as a hybrid genetic algorithm combined with some kinds of local search''}, and can be more specifically classified as a hybrid nature-inspired algorithm \cite{122}. This type of algorithms have been previously used, with success, for optimization problems \cite{123} in general, and for integer linear programming problems \cite{124} in particular. Thus, a secondary novel aspect of the present research, in addition to devolutionary computation, lies in the use of a hybrid meta-heuristic approach based on an integer linear program formulation of the MLST problem and the introduction of a new class of valid constraints for Steiner problems. A key focus in memetic computing being the inclusion of problem knowledge into the solver technique \cite{125}, we will aim at effectively making use of these constraints in the algorithm, to guide the search procedure and fasten convergence.

\section{Integer Linear Programming Formulation}
\label{formulation}
Similarly to Beasley's formulation of the Steiner tree problem \cite{17}, the MLST problem can be stated as finding a minimum labeling spanning tree $T'$ in a modified network $G' = (V',E',L)$, generated by adding a new node $v'$, and connecting it using ``colorless'' edges to all nodes in $V \backslash Q$ and to an arbitrarily fixed terminal $q_0$, with additional constraints stating that every node in $V \backslash Q$ that is adjacent to $v'$ in $T'$, must be of degree one. As the adjective suggests, we consider ``colorless'' edges to be edges whose labels are not counted, when evaluating a tree they are part of. \\

Let edge variables $x_e \in \{0, 1\}, \forall e \in E'$ and label variables $y_l \in \{0, 1\}, \forall l \in L$ respectively indicate whether an edge $e$ and a label $l$ are used in a spanning tree corresponding to a solution to the MLST problem. We denote $X$ the solution-vector constituted of variables $x_e, \forall e \in E'$, $\delta(k)\subseteq E'$ the set of edges that are incident to a node $k, \forall k \in V'$ or, by extension, the set of edges that have exactly one endpoint in a subset of nodes, and $l(e)$ the color of edge $e, \forall e \in E$. Note that it is not necessary to model the labels of edges in $E' \backslash E$, i.e. ``colorless'' edges. \\

A generic formulation of the problem is given by the following integer linear program:
\begin{alignat}{4}
&\min \sum \limits_{l \in L} y_l & \\
\text{s.t. }& X\text{ is a Spanning tree}&  \\
&y_{l(e)} \geq x_e & \forall e \in E \\
&x_{\{v',k\}} + x_{\{k,i\}} \leq 1 &  \forall k \in V \backslash Q, i\in V'\backslash \{v'\}: \{k, i\} \in \delta(k) \\
&x_e \in \{0, 1\} & \forall e \in E' \\
&y_l \in \{0, 1\} & \forall l \in L 
\end{alignat}
In this linear program, the objective function $(1)$ minimizes the number of required labels. The abstract constraint $(2)$ simply states that the resulting solution $X$ constitutes a Steiner tree. There exist numerous ways to make this constraint explicit. In the following, we assume that it is replaced by the two following types of inequalities, where $E(W)$ is the set of edges with both endpoints in a subset of nodes $W \subset V'$
\begin{alignat}{3} 
\sum \limits_{e \in E'} x_e = n  &  &  \\
\sum \limits_{e \in E':e\in E(W)} x_e & \leq |W|-1  \quad  & \Phi \neq W \subset V'  
\end{alignat}
Note that the number of constraints of type $(8)$ is exponential in the number of nodes of the graph. The two main challenges in solving this model are thus the exponential number of constraints and the integer nature of variables.\\

Inequalities $(3)$ ensure that the label variable associated with the label of each edge is equal to 1, if said edge is part of the solution. Inequalities $(4)$ enforce the previously-described degree constraints on nodes from $V \backslash Q$ that are adjacent to $v'$ in $T'$.  Finally, constraints $(5)$ and $(6)$ ensure that all variables are binary.\\

The number of constraints in the linear program can be reduced by replacing inequalities $(3)$ with the following, in which $\left\vert{S_l}\right\vert$ denotes the cardinality of the subset $S_l \subset E$ of edges whose label is $l$: 
\begin{alignat}{3}
 \sum \limits_{e \in E:l(e)=l} x_e \leq min\{\left\vert{S_l}\right\vert, n-1\} \cdot y_l \quad &   \forall l \in L
\end{alignat}

\section{Proposed Devolutionary Approach}
\label{proposed}
The proposed devolutionary approached can be outlined as follows:
\begin{enumerate}[label=\Alph*]
\item Generate a population of super-optimal solutions.
\item Evaluate each individual's fitness and determine population's average fitness.
\item Repeat
\begin{itemize}         
\item Select best-ranking individuals to reproduce
\item Mate at random
\item Apply crossover operator
\item Apply mutation operator
\item Evaluate each individual's fitness
\item Determine population's average fitness
\end{itemize}
Until the desired number of feasible solution is reached, in which case a feasible solution cannot  reproduce anymore. 
\item Choose the best feasible solution.
\end{enumerate}

In the case of the MLST problem, the generation of the initial population, the crossover and mutation operators as well as the evaluation of fitness are performed as follows.

\subsection{Generating a population of initial solutions}
This is performed by relaxing integrity constraints in the integer linear programming formulation of the MLST problem and generating a subset of optimal solutions to the relaxed problem, or alternatively super-optimal solutions to the MLST problem.\\

As stated before the number of inequalities of type (8) is too high to deal with
all of them right from the start of the optimization process. Both the exact method we shall use for comparison and the meta-heuristic method presented here start with a reduced ILP formulation. All the presented inequalities get added except
constraints (8). An iterative process is started that consists in solving this reduced
linear program, separating violated constraints, adding them to the
problem and restarting with solving the enhanced problem. The separation
of violated inequalities, which is not the main focus of this paper, is done using a generic separation library under A Branch-And-CUt System (ABACUS) \cite{ABACUS}. \\

Various super-optimal solutions can be generated by varying the choice of the arbitrary terminal to which $v'$ is connected. For some instances, it can also be fruitful to generate this initial population by using a branching procedure on the values of a small subset of binary variables (e.g. a subset of labels). 
\subsection{Fitness function}
An exploitable property of the present formulation is that if edge variables $x_e, \forall e \in E'$ take an integer value in a solution, label variables $y_l, \forall l \in L$ would also have an integer value, given the structure of constraints of type (3) or (9) and the fact that objective function (1) minimizes the sum of the latter type of binary variables. Thus, one can focus on increasing the number of edge variables that take an integer value to tend towards feasibility. For this reason, the fitness function evaluates the feasibility of a solution $X$ as the number of edge variables $x_e, \forall e \in E'$ that have an integer value in the solution-vector. It is formally calculated as 
$\left\vert\{x_e \in X: x_e \in \{0, 1\} \}\right\vert$. Once a solution in which all edge variables take integer values is reached, this solution would be feasible for the MLST problem.

\subsection{Crossover and mutation}
In this section we introduce a new class of Chv\`atal-Gomory cutting planes.  As opposed to the constraints presented in the previous sections, this class of inequalities is not needed to model the MLST problem as an integer linear program. However, their addition can reduce the solving time needed, by cutting the polyhedron which is already defined by inequalities (2) to (9). After proving the validity of these inequalities, we show how they can be used in the crossover operation, in order to speed up the convergence process. \\

The following example shows the type of fractional solutions that can result from relaxing integrity constraints in the previous integer linear program. In Figure \ref{example}, we consider a sub-graph including four terminals denoted $\{t_1,t_2,t_3, t_4\}$ and two Steiner nodes denoted $\{s_1, s_2\}$, to which a node $v'$ is added and connected to all Steiner nodes and to terminal $t_1$, as per the previously described formulation of the problem. \\

An integer solution, in which variables corresponding to edges $\{s_1,v'\},$  $\{v',t_1\},$  $\{t_1,t_2\},$  $\{s_2,t_2\},$  $\{t_2,t_3\}$ and $\{t_3,t_4\}$ take value $1$, while variables corresponding to all other edges of the sub-graph take value $0$, as represented in Figure \ref{example2}, is feasible and represents a Steiner tree in this sub-graph. We modify this solution by reducing the value of the variable corresponding to edge $\{t_1,t_2\}$ to $0.5$ and assigning another $0.5$ to the variable corresponding to edge $\{s_1,s_2\}$, as represented in Figure \ref{example3}. It is easy to see that this results in a fractional solution that would satisfy constraints (2) to (6) and be feasible for the relaxed problem. \\

In order to cut such solutions off from the polyhedron of solutions to the relaxed problem, we describe a set of constraints that can be imposed on a subset of elementary cycles passing through node $v'$ in network $G'$. Node $v'$ being connected to all Steiner node and to one terminal,  this node is either adjacent to two Steiner nodes, or to one Steiner node and one terminal, in all cycles passing through $v'$. We focus on all elementary cycles passing through $v'$, that are of the former type. It should be noted that for each elementary path between two Steiner nodes in $G$, corresponds such a cycle in $G'$. \\

We first consider non-triangle cycles and then treat the case of triangles separately. Let us denote by $C$ a non-triangle cycle in $G'$ passing through $v'$ which contains two edges $\{v',s_i\}$ and $\{v',s_j\}$, where $s_i$ and $s_j $ are two Steiner nodes in $V \backslash Q$, by $x_{\{s_i,s_j\}}$ the variable corresponding to edge $\{s_i,s_j\}$, and by $\left\vert{C}\right\vert \geq 4$ the length (number of edges) of $C$. We introduce the following inequalities:
\begin{alignat}{3} 
\sum \limits_{e \in C} x_e \leq \left\vert{C}\right\vert -2 - x_{\{s_i,s_j\}} & \forall C \mbox{ a cycle}  : \{v',s_i\},\{v',s_j\} \in C, \{s_i,s_j\} \not \in C,  s_i, s_j \in V \backslash Q 
\end{alignat}

In Figure \ref{example}, one can observe for instance that cycle $v' s_2 t_2 t_3 t_4 s_1 v'$, of length 6, would result in the following inequality: 
$$x_{\{v', s_2\}}+x_{\{s_2, t_2\}}+x_{\{t_2, t_3\}}+x_{\{t_3, t_4\}}+x_{\{t_4, s_1\}}+x_{\{ s_1, v'\}}\leq 4 - x_{\{s_2, s_1\}}$$

\begin{figure}[!htb]
\minipage{0.32\textwidth}
\caption{\label{example} An illustrative graph with four terminals and two Steiner nodes}
\includegraphics[width=\linewidth]{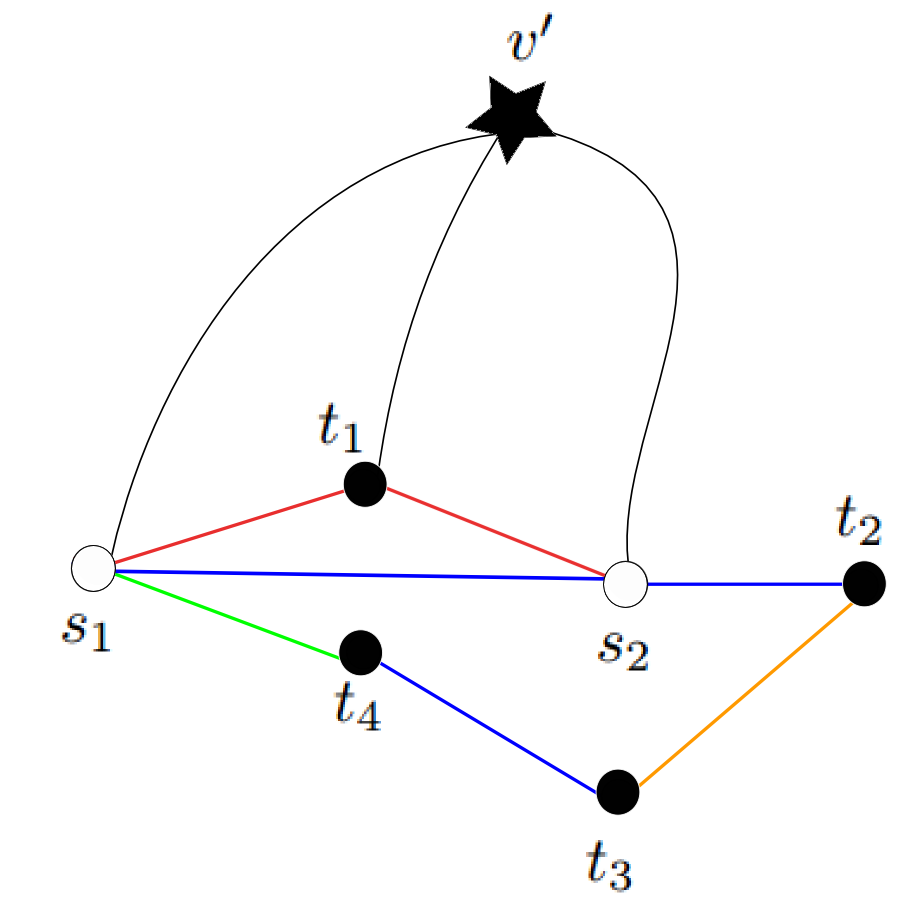}
\endminipage\hfill
\minipage{0.32\textwidth}
\caption{\label{example2} A feasible integer solution preserved by inequality (10)} 
\includegraphics[width=\linewidth]{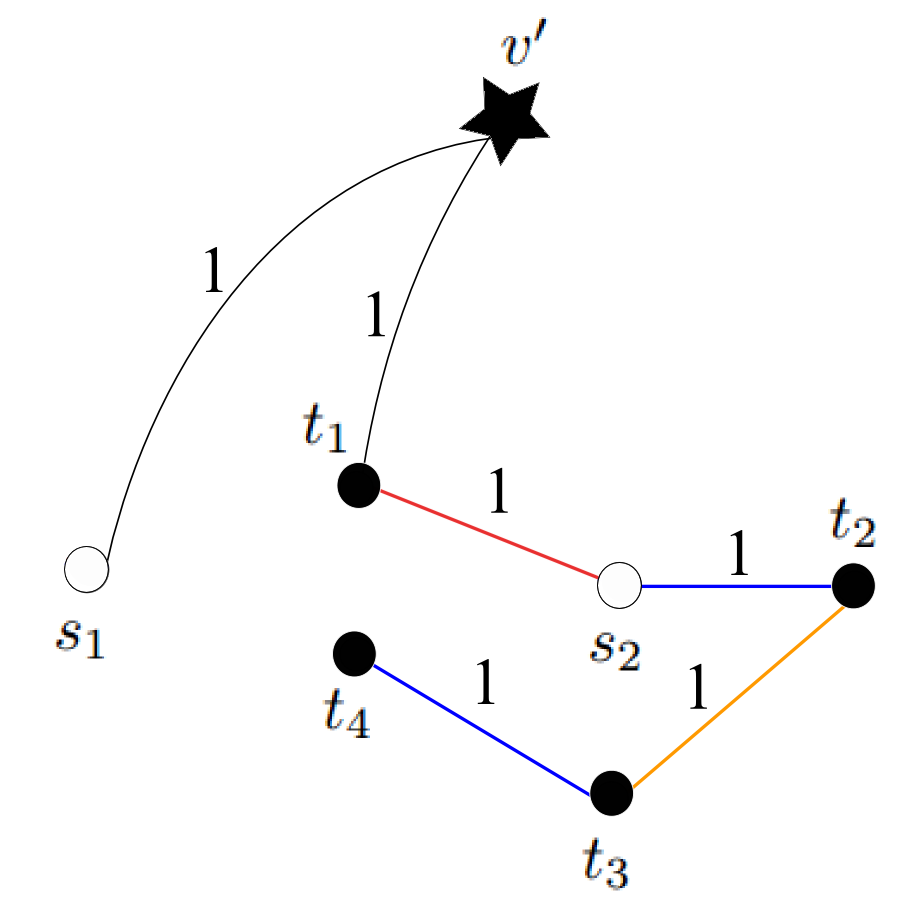}

\endminipage\hfill
\minipage{0.32\textwidth}%
\caption{\label{example3} An unfeasible fractional solution cut off by inequality (10)} 
\includegraphics[width=\linewidth]{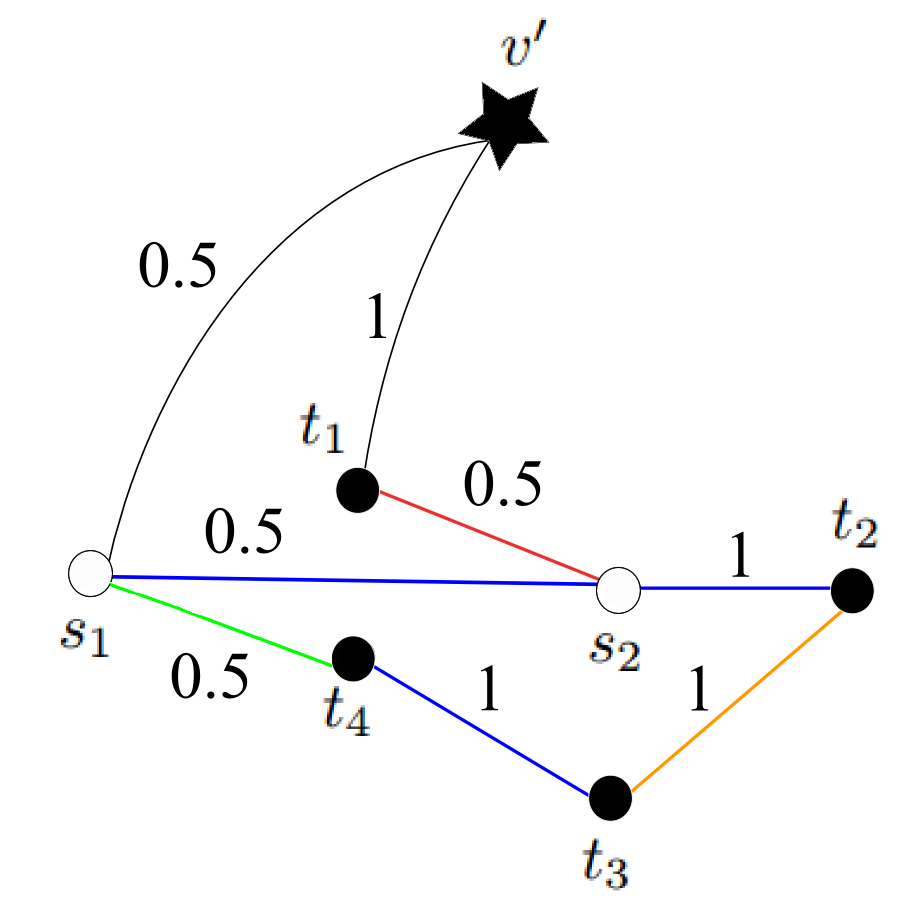}

\endminipage
\end{figure}

When applied to the integer solution represented in Figure \ref{example2}, this inequality would be tightly satisfied (4 $\leq$ 4), while it would not be satisfied (4 $\not\leq$ 3.5) by the fractional solution represented in Figure \ref{example3}.

We extend this results to all solutions in Proposition 1, a proof of which is provided below.

\begin{proposition}
Constraint (10) is a valid inequality for the set of feasible solutions to the MLST problem, defined by constraints (2) to (6).
\end{proposition}

\begin{proof}
We show that constraints (10) would be satisfied by any Steiner tree in $G$. Let $T'$ be a spanning tree in $G'$ that corresponds to a Steiner tree in $G$, that is to say that every node in $V \backslash Q$ that is adjacent to $v'$ is of degree one in $T'$. We consider $C$ to be a cycle of length $\left\vert{C}\right\vert \geq 4$ in $G'$ that contains two edges $\{v',s_i\}$ and $\{v',s_j\}$, with $s_i, s_j \in V \backslash Q$. \\

Since $T'$ is a spanning tree in $G'$, it cannot contain any cycle, thus the number of edges in $C \cap T'$ is at most $\left\vert{C}\right\vert -1$. Therefore $\sum \limits_{e \in C} x_e \leq \left\vert{C}\right\vert -1$ holds. Moreover, the degree constraint on Steiner nodes in  $V \backslash Q$ imposes that the degrees of $s_i$ and $s_j$ be equal to one in $T'$, which has two consequences:
\begin{itemize}
\item If $\left\vert{C}\right\vert \geq 4$, the number of edges in $C \cap T'$ cannot be equal to $\left\vert{C}\right\vert -1$, because that would imply that one node among $s_i$ and $s_j$ has a degree at least equal to two in $T'$. Thus $\sum \limits_{e \in C} x_e \leq \left\vert{C}\right\vert -2$ holds.\\

\item If the number of edges in $C \cap T'$ was equal to $\left\vert{C}\right\vert -2$, we show that edge $\{s_i,s_j\}$ would not be part of $T'$, i.e. $x_{\{s_1,s_2\}}=0$. Firstly, edge $\{s_i,s_j\}$ being part of $T'$ would imply that edges $\{v',s_i\}$ and $\{v',s_j\}$ are not part of $T'$, otherwise $T'$ would contain triangle $v' s_i s_j v'$. Thus if edge $\{s_i,s_j\}$ is part of $T'$, then exactly one edge among $\{v',s_i\}$ and $\{v',s_j\}$, and exactly one edge among the two edges in $C\cap  (\delta(s_i) \backslash \{v',s_i\}\cup \delta(s_j) \backslash \{v',s_j\})$ (i.e. the edge that comes after $\{v',s_i\}$ or the edge that comes after $\{v',s_j\}$ when going through cycle $C$ from node $v'$) are left out in the construction of $T'$, and these two edges cannot be incident to the same node. Without loss of generality let us consider that the number of edges in $C \cap T'$ is equal to $\left\vert{C}\right\vert -2$, and that the two edges in $C$ that are not part of $T'$ are $\{v',s_i\}$ and an edge $e_j \in C\cap (\delta(s_j) \backslash \{v',s_j\})$. If edge $\{s_i,s_j\}$ was part of $T'$ then node $s_i$ would have a degree at least equal to two, which violates the degree constraint on this node. Therefore, inequality $ x_{\{s_i,s_j\}} \leq \left\vert{C}\right\vert -2 - \sum \limits_{e \in C} x_e$ forcing $ x_{\{s_i,s_j\}}$ to take value 0 if the number of edges in $C \cap T'$ equals $\left\vert{C}\right\vert -2$, holds, which is an equivalent way to state inequality (10). 
\end{itemize}
Therefore, constraint (10) is satisfied by any Steiner tree in $G$. The previous example additionally, showed that inequality (10) is not satisfied by some fractional solutions that would otherwise satisfy constraints (2) to (6), and that this inequality is tightly satisfied by some solutions corresponding to Steiner trees. Thus constraint (10) is a valid inequality for the MLST problem. 
\end{proof}

In the case of triangles of the form $v' s_i s_j v'$, we should mention that constraints of type (10) would exclude integer solutions where $x_{\{v',s_i\}}=x_{\{s_j, v'\}}=1$ and $x_{\{s_i,s_j\}}=0$, which correspond to Steiner nodes $s_i$ and $s_j$ not being used in the corresponding tree. Indeed, in this case the number of edges in $C \cap T'$ can be equal to $\left\vert{C}\right\vert -1=2$. The following constraint can however be stated:

\begin{alignat}{3}
\frac{x_{\{v',s_i\}} + x_{\{v',s_j\}}}{2} \leq 1 - x_{\{s_i,s_j\}} \quad & \forall C \mbox{ a triangle }  v' s_i s_j v': s_i, s_j \in V \backslash Q  
\end{alignat}

This constraint, whose validity is easy to verify, is an adaptation of constraints (10) to triangles. It imposes that if edges $\{s_i,s_j\}$ is part of a Steiner tree $T'$, then edges $\{v',s_i\}$ and $\{v',s_j\}$ cannot be part of it.\\

Constraints (10) and (11) can be used in the crossover operation. Let $X_1$ and $X_2$ be two fractional solution vectors corresponding to super-optimal solutions, the crossover operation is performed in the following two steps:

\subsubsection{Passing on good features}
The crossover operation first crosses the two sets of colors used by edges incident to each terminal $t_i\in Q$, whose corresponding edge variables take value 1, as well as the two sets of such edges.\\

It should be noted that this terminal-by-terminal crossover is a different operation from only crossing the sets of colors used by common edges or crossing the sets of all colors used in each solution, as illustrated in Figure \ref{example4}, representing two subsets of arcs whose value is 1 in two super-optimal solutions over the same sub-graph. This sub-graph contains two terminals $\{t_1, t_2\}$ and four Steiner nodes $\{s_1, s_2, s_3, s_4\}$, and edges in these two solutions are colored in \textit{red} (R), \textit{green}, \textit{blue} (B) or \textit{yellow} (Y).\\

Crossing these two solutions leads to the yellow color being passed on to the offspring, as edge $\{t_1, s_3\}$ uses this color in both solutions, to the blue color being passed on to the offspring, although it is used in two different edges incident to  $s_1$, but not the red color as it is used in edges adjacent to two different terminals, and obviously not the green color as it is used in only one of the two solutions. Thus, crossing these two solutions would lead to the values of variables $y_l , l \in \{\mbox{B}, \mbox{Y}\}$, as well as that of edge variable corresponding to $\{t_1, s_3\}$ being set to 1 in the offspring. 

\begin{figure}[!htb]
\begin{center}

\caption{\label{example4} An illustration for the first phase of the crossover operation}
\includegraphics[width=0.7\linewidth]{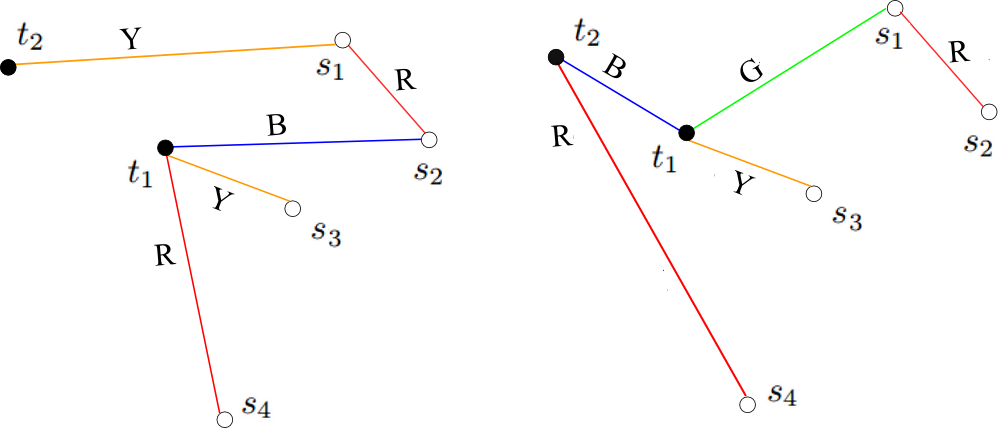}
\end{center}
\end{figure}

Formally we consider the following set of edge variables for each terminal $t_i \in Q$: 
$Z_i=\{e \in  \delta(t_i): y_{l(e)}=1, \mbox{ in both }\ X_1 \mbox{ and } X_2\}$ and set $y_{l(e)}=1, \forall e\in Z_i$ as well as $x_e=1, \forall e\in Z_i: x_e=1, \mbox{ in both }\ X_1 \mbox{ and } X_2$. \\

The intuitive idea of this procedure is to pass on the common colors and edges that are used to connect each terminal, in solutions $X_1$ and $X_2$ to their offspring. The second step of the crossover procedure aims at the progressive removal of fractional-valued variables.
\subsubsection{Cutting off bad features}
In this second step, we consider edge variables that have a fractional value in $X_1$ and $X_2$, specifically those corresponding to edges that are incident to one or two Steiner nodes, and impose a constraint of type (10) or (11) over a cycle containing each one of them. This procedure can be formally stated as follows:
$\forall e \in  E': x_e \mbox{ is fractional in }\ X_1 \mbox{ or } X_2\}$:
\begin{itemize}
\item  If only one extremity of $e$ is a Steiner node, identify an elementary path between the terminal extremity of $e$ and another Steiner node, using a Depth First Search, and impose a constraint of type (10) over the non-triangle cycle passing through $v'$ thus constituted.
\item  If both extremities of $e$ are Steiner nodes, impose a constraint of type (11) over the triangle formed by $v'$ and the two extremities of $e$.
\end{itemize}
Once this procedure is performed, the resulting simplified linear program is solved and new solutions presenting the highest levels of fitness are subjected to the same procedure. Additionally, a mutation operator periodically selects a terminal and allows it to be connected using a previously unused color.

Given a graph $G' = (V',E',L)$, and two fractional solutions $X_1$ and $X_2$, the crossover operation, that would generate an offspring solution $X_3$, can be summarized by Algorithm \ref{cross}. 
\begin{algorithm}[t]

 \KwData{$G' = (V',E',L)$, $X_1$, $X_2$ }
 \KwResult{$X_3$}
\For {all   $t_i \in Q$}
{Define $Z_i=\{e \in  \delta(t_i): y_{l(e)}=1, \mbox{ in both } X_1 \mbox{ and } X_2\}$\\
\For{all $e\in Z_i$ }
{
\If{$x_e=1, \mbox{ in both } X_1 \mbox{ and } X_2$}
{Set $x_e=1$ in the relaxed linear program\\
Set $y_{l(e)}=1$ in the relaxed linear program\\
}
}
  }
\For {all   $e \in  E'$}
{\If{$x_e \mbox{ is fractional in } X_1 \mbox{ or } X_2$}
{Define $v_1, v_2=\mbox{ the two extremities of } e$\\
\If{$v_1 \in V \backslash Q$ and $v_2 \in V \backslash Q$ }
{Add constraint $\frac{x_{\{v',v_1\}} + x_{\{v',v_2\}}}{2} \leq 1 - x_{\{v_1,v_2\}}$ to the relaxed linear program\\ 
\textbf{else} \If{$v_1 \in V \backslash Q$ and $v_2 \in Q$}
{\For {all   $s \in  Q\backslash v_2$}
{Define cycles $C=\{v',s\}\cup e\cup DFS(v_1,s)\cup \{v',v_2\}$\\
Add constraint $\sum \limits_{a \in C} x_a \leq \left\vert{C}\right\vert -2 - x_{\{v_1,v_2\}}$ to the relaxed linear program

}

}
}
}
}

Define $X_3=$ Optimal solution of the relaxed linear program\\

 \caption{\label{cross}Crossover procedure}
\end{algorithm}
\subsection{Illustrative example}
To illustrate the functioning of the proposed algorithm, let us consider the graph provided in Figure \ref{toy}, which contains three terminals denoted $\{t_1, t_2, t_3\}$ and three Steiner nodes denoted $\{s_1, s_2, s_3\}$, in addition to node $v'$ that is connected to all Steiner nodes and to terminal $t_2$ using colorless edges (labeled C), as per our formulation. Other edges in the graph can be colored in \textit{red} (R), \textit{green} (G), \textit{blue} (B) or \textit{yellow} (Y). Figure \ref{toy1} and Figure \ref{toy2}, in which edges are labeled according to the non-zero values of their corresponding edge variables in the linear program and omitted if those variables take a zero value, represent two super-optimal fractional solutions generated by solving the relaxed linear program corresponding to this graph. We respectively denote these two solutions $X_1$ and $X_2$.
One can observe that $X_1$ and $X_2$ are both of value $1.5$ and exhibit the typical structure of a super-optimal fractional solution in our formulation. Indeed, some edges are selected and their corresponding label variable   $y_l, l \in L$ take value 1 (e.g. the green color in our two super-optimal solutions), While other edges take a fractional value and thus their corresponding label variables $y_l, l \in L$ take a fractional value. This is the case for the red color in the super-optimal solution in Figure \ref{toy1}, and for the blue color in the super-optimal solution in Figure \ref{toy2}. The goal of the crossover procedure is, simply put, to exploit this common structure for super-optimal solutions, by passing on a subset of the former type of edges, while cutting off the latter type of edges. Thus, as per Algorithm \ref{cross}, crossing $X_1$ and $X_2$ leads to selecting common edges $\{t_1, t_2\}$ and $\{t_3, s_2\}$ that are incident to at least one terminal. Figure \ref{toy3} represent the selected edges after this phase of the crossing procedure is performed. Furthermore, a depth-first search from edge $\{s_3,t_3\}$, whose corresponding edge variable has a fractional value in $X_1$ identifies the cycle $C_1=v' s_3 t_3 s_2 v'$ of length 4. A similar search from edge $\{s_1,t_3\}$, whose corresponding edge variable has a fractional value in $X_2$, identifies the cycle $C_2=v' s_1 t_3 s_2 v'$, also of length 4. Thus, the two following inequalities of type (10) are to be imposed respectively over $C_1$ and $C_2$, in the relaxed linear program:
$$x_{\{v',s_3\}} + x_{\{s_3,t_3\}} + x_{\{t_3,s_2\}} + x_{\{s_2,v'\}}  \leq 4 -2 - x_{\{s_3,s_2\}} $$
$$x_{\{v',s_1\}} + x_{\{s_1,t_3\}} + x_{\{t_3,s_2\}} + x_{\{s_2,v'\}}  \leq 4 -2 - x_{\{s_1,s_2\}} $$
It is easy to observe that solution $X_1$ does not satisfy the first inequality (1.5 $\not\leq$ 1), while solution $X_2$ does not satisfy the second inequality (2 $\not\leq$ 1.5). Therefore, these two super-optimal fractional solutions would be cut off by these the two constraints. 

Finally, solving the relaxed linear program under these conditions generates the integer solution $X_3$ represented in Figure \ref{toy4}, which for our example, constitutes an optimal Steiner tree of value 2.
\begin{figure}[!htb]
\minipage{0.32\textwidth}
\caption{\label{toy} An illustrative graph with three terminals and three Steiner nodes}
\includegraphics[width=\linewidth]{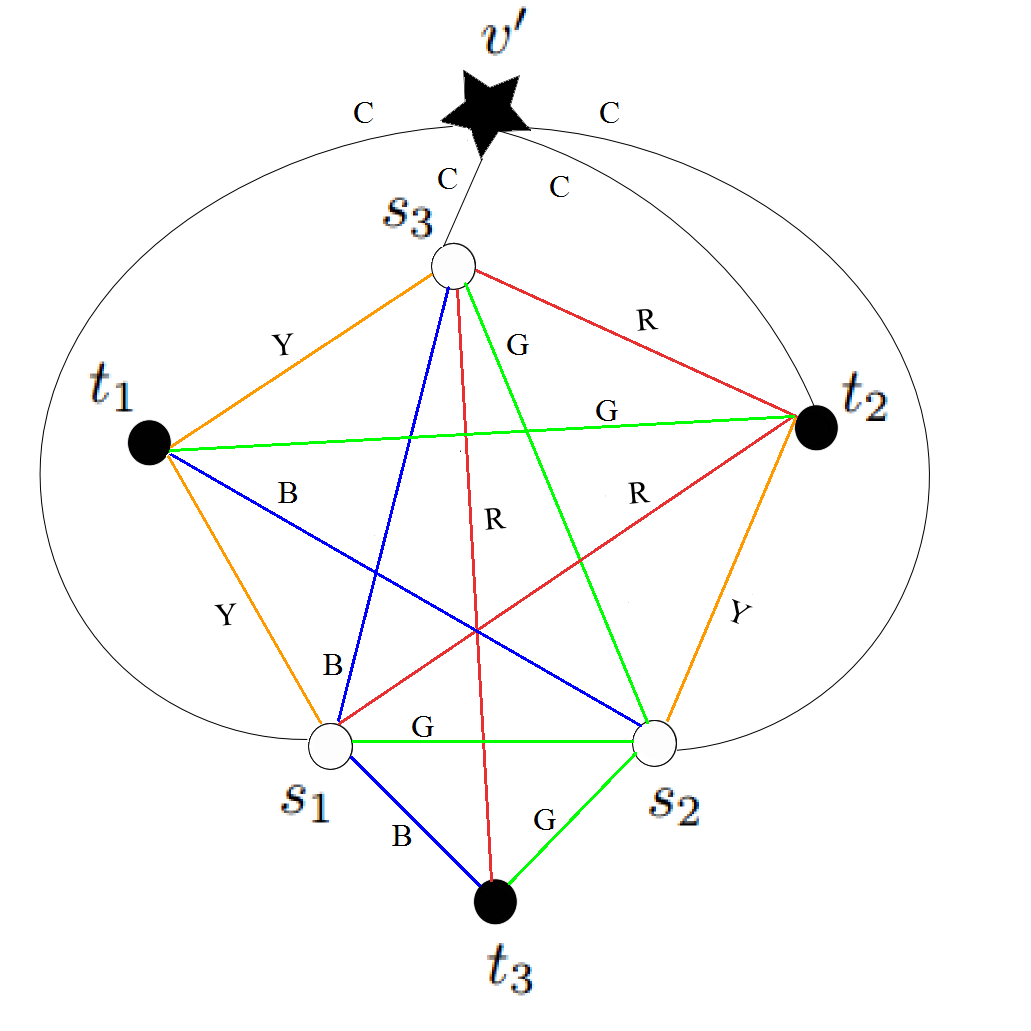}
\endminipage\hfill
\minipage{0.32\textwidth}
\caption{\label{toy1}First super-optimal fractional solution ($X_1$)} 
\includegraphics[width=\linewidth]{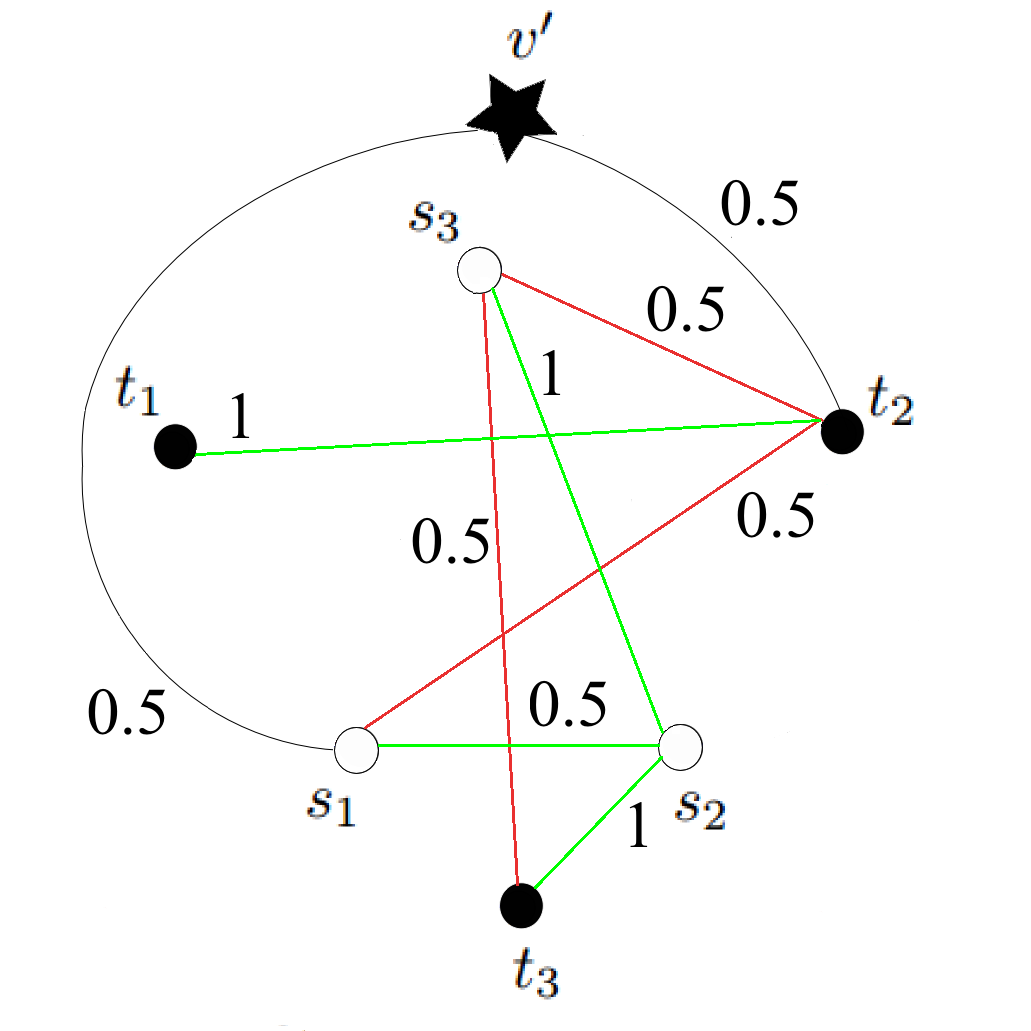}

\endminipage\hfill
\minipage{0.32\textwidth}%
\caption{\label{toy2}Second super-optimal fractional solution ($X_2$)} 
\includegraphics[width=\linewidth]{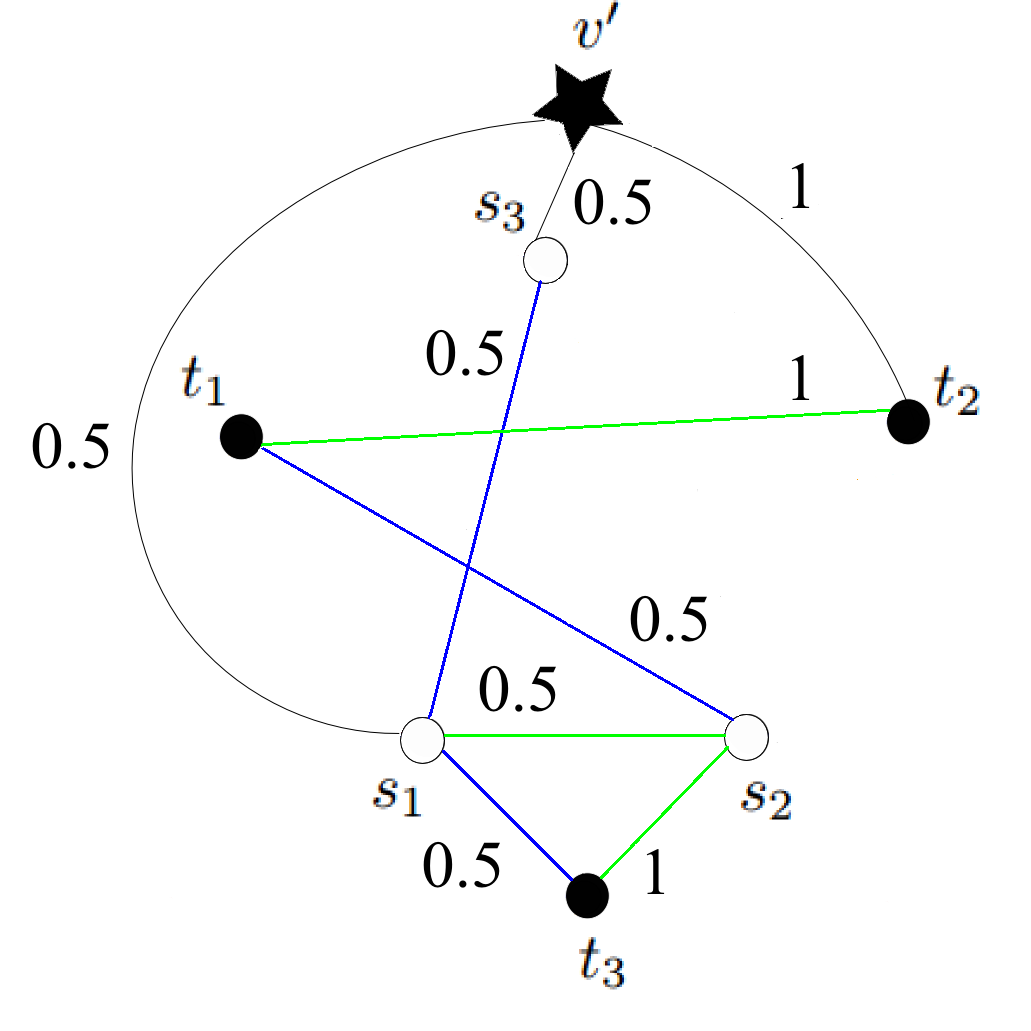}
\endminipage
\end{figure}
\begin{figure}[!htb]
\minipage{0.05\textwidth}
\hphantom
\endminipage\hfill
\minipage{0.32\textwidth}%
\caption{\label{toy3} Edges selected by crossing the two super-optimal solutions}
\includegraphics[width=\linewidth]{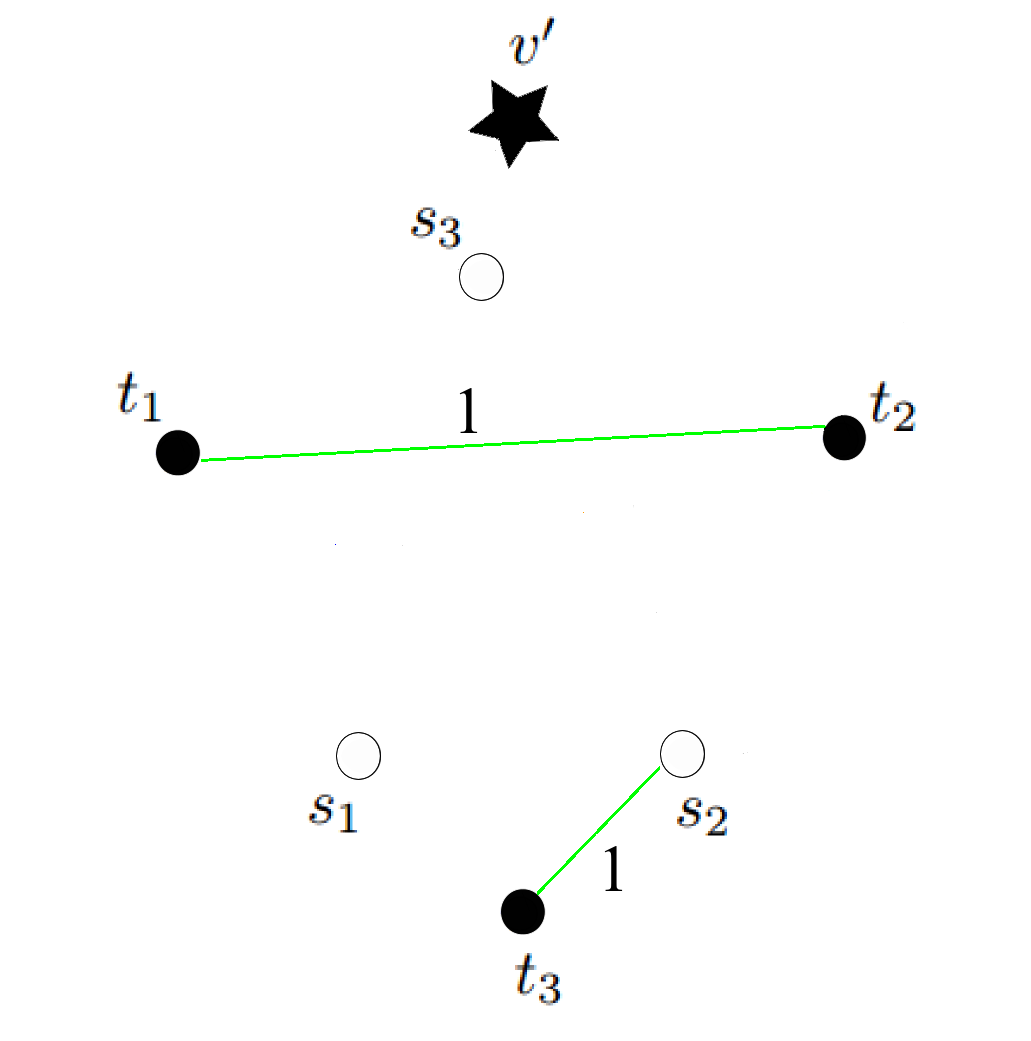}

\endminipage\hfill
\minipage{0.32\textwidth}%
\caption{\label{toy4} Optimal Steiner tree generated through the crossover operation ($X_3$)} 

\includegraphics[width=\linewidth]{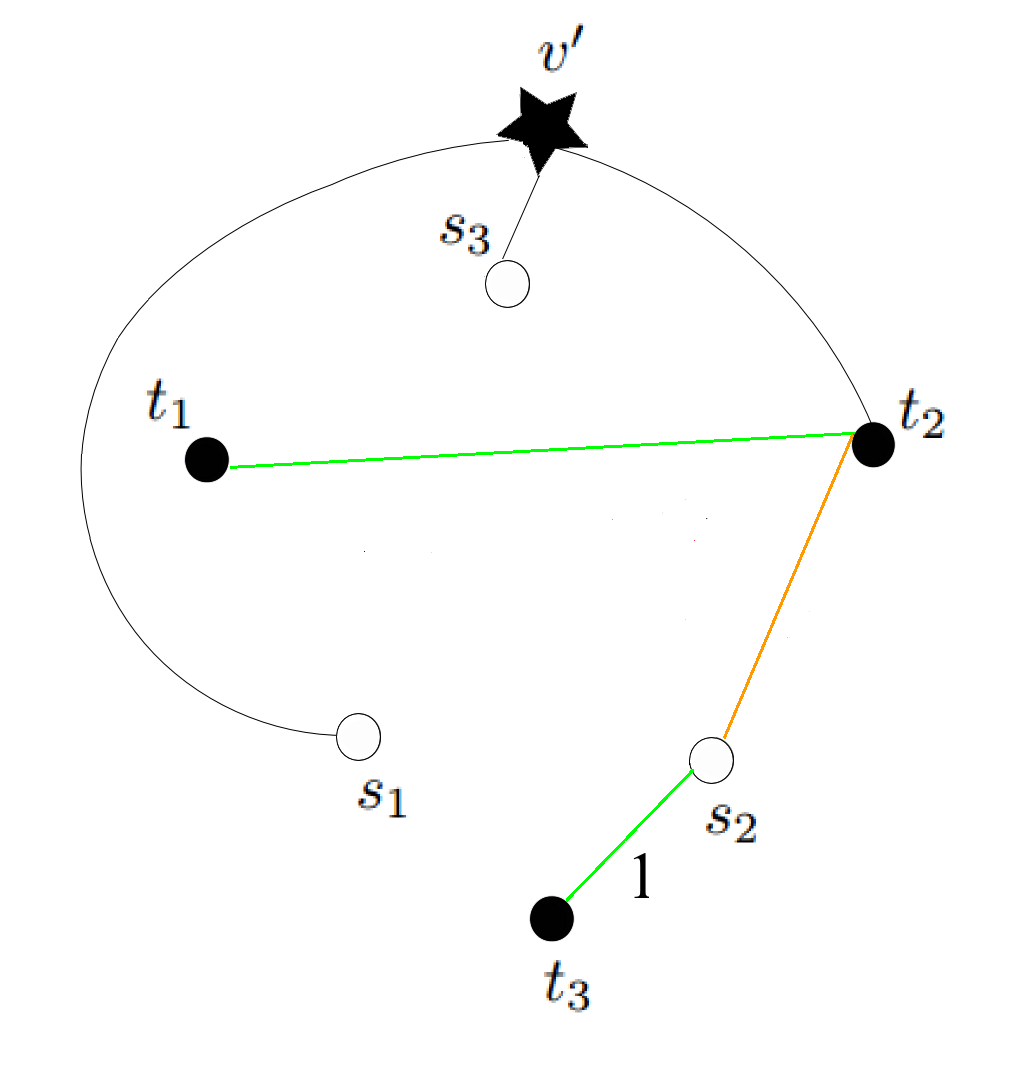}
\endminipage\hfill
\minipage{0.05\textwidth}
\hphantom
\endminipage
\end{figure}

\section{Application and experimental results}
\label{computation}
\subsection{Test instances}
We have conducted preliminary experiments in which different C++ implementations of the proposed generic algorithmic approach for the MLST problem were compared in terms of solution quality and computational running time. The best-performing implementation of the devolutionary genetic algorithm (DGA) in our experiment was compared to an exact branch and bound method (BB) and to the Pilot Method (PM). We considered $20$ different randomly generated data-sets, each containing $5$ instances of the problem, resulting in a total of a $100$ instances, with $n \in \{50,\dots,200\}$ nodes, a number of basic nodes $|Q|=0.25\cdot n$, values of $m$ derived from densities $ d \in \{0.25, 0.5, 0.8\}$, and a number of colors $c \in \{0.25\cdot n, 0.50\cdot n, 1.00\cdot n\}$. The three algorithms (DGA, BB and PM) were ran once for each instance. All computations have been conducted on an Intel Core i7 processor at $8 \times 2.20 GHz$ with $8 Gb$  Ram. Computation duration for PM and DGA were calculated by recording the time at which the best solution found was first discovered, notwithstanding the additional time in which the algorithms ran without improvement to this solution. The GNU Linear Programming Package (GLPK) was used for solving integer linear programs and their relaxations.\\ 

It should be noted that a comparison with an established evolutionary genetic algorithm for the problem at hand would have been informative as well. However, no such method has been previously reported in the literature, to the best of our knowledge. Moreover, developing an ad-hoc evolutionary genetic algorithm in this work, in addition to being a challenging task in its own right, given the sensitivity of solving the MLST problem to the initial random setting of labels, as reported in \cite{1}, would have defeated the purpose of this experiment, which was to test the performances of the devolutionary approach by comparison with established methods for solving the MLST problem.    
\subsubsection{Comparison results}
\begin{table}[!htb]
      \centering
    \begin{minipage}{.4\linewidth}
  
    \centering
 \caption{\label{res1} Average objective function value for $|Q|=0.25\cdot n$}
        \begin{tabular}{|c|c|c||c|c|c|}
        \hline
       $n$& $d$ & $c$ & PM & DGA & BB\\
      \hline
          &   & $12$ &2.6  & 2.8 & 2.6\\
        & 0.25& $25$ &2.8 & 3 & 2.8\\
         &    & $50$ & 3.1 & 3.3 & 3.1\\
                   \cline{2-6}
          &  & $12$ & 1.5 & 1.25 &1.25 \\
  50  & 0.50& $25$ & 1.7 & 1.41 & 1.4\\
         &    & $50$ &2.2  & 2.2 &2.2 \\
             \cline{2-6}
         &    & $12$ & 1.4   & 1.2 &1.2 \\
        & 0.80& $25$ & 1.35 & 1.35 & 1.35\\
        &     & $50$ & 1.3 & 1.3 & 1.3\\
             \hline   
             
                &    & $25$ & 9.6   & 9.2 &8.5 \\
        & 0.25& $50$ & 10.65 & 9.0 & 8.0\\
        &     & $100$ & 12.3 & 10.5 & 9.2\\

                   \cline{2-6}
  &   & $25$ &8.4  & 8.0 & 8.0\\
 100    & 0.50 &$50$ &8.9 & 8.5 & 7.8\\
         &    & $100$ & 10.3 & 9.3 & 9\\
             \cline{2-6}
         &  & $25$ & 8.0 & 6.2 &7.5 \\
   & $0.80$ &$50$ & 8.7 & 7.4 & 7.4\\
         &    & $100$ &9.2  & 9.2 &8.2 \\
             \hline   
      &    & $50$ & 11.6   & 11.2 &11.0 \\
        & 0.25& $100$ & 18.2 & 13.8 & 12.35\\
        &     & $200$ & 20.4 & 19.0 & 17.3\\
       
                   \cline{2-6}
          &  & $50$ & 10.5 & 10.2 &8.4 \\
  200  & 0.50& $100$ & 12.6 & 11.9 & 10.0\\
         &    & $200$ &16.8  &15.0 &14.6 \\
             \cline{2-6}

          &   & $50$ &7.6  & 7.8 & 7.8\\
        & 0.80& $100$ &8.8 & 8.3 & 8.5\\
         &    & $200$ & 9.1 & 9.1 & 9.0\\
             \hline   
              
        \end{tabular}

    \end{minipage}%
    \hspace{.1\linewidth}
    \begin{minipage}{.4\linewidth}
      \centering
       \caption{\label{res2} Average computation duration (in seconds) for  $|Q|=0.25\cdot n$}
        
       \begin{tabular}{|c|c|c||c|c|c|}
        \hline
    $n$& $d$ & $c$ &  PM & DGA &BB\\
       \hline
         &    & $50$   & 2 & 3&4.8\\
        & 0.25& $100$  & 3 & 3.5&6.3\\
         &    & $200$  & 5 & 8.2&12\\
          \cline{2-6}
         &   & $50$  & 1.7 &3 & 3.5\\
  50      & 0.50& $100$  & 2.3 & 3.2& 4.3\\
        &     & $200$   & 5.5 &7.5&9 \\
        \cline{2-6}
        &     & $50$    & 1.5 &1.5 & 2.4\\
        & 0.80& $100$  & 2 & 3.0& 3.2\\
        &     & $200$  & 3.2 & 3.5& 5\\
             \hline
             
         &     & $50$    & 21.5 &112.7 & 162.7\\
        & 0.25& $100$  & 18 & 136.1& 229.2\\
        &     & $200$  & 27 & 173.3& 300.5\\

          \cline{2-6}
         &   & $50$  & 5.6 &7.6& 13.5 \\
100        & 0.50& $100$  & 11.3 & 36.2& 56.3\\
        &     & $200$   & 15.2 &53.5&89 \\
        \cline{2-6}
           &    & $50$   & 4.4 &4.8 &9\\
       & 0.80& $100$  &6.3 & 7.2&13.5\\
         &    & $200$  & 11.1 & 13.4&20.9\\
             \hline

             &    & $50$   & 32.3 & 132.6&400.8\\
      & 0.25& $100$  & 40.7 & 142.6&1014.0\\
         &    & $200$  & 51.7 & 260.2&Unknown\\
          \cline{2-6}
         &   & $50$  & 12.6 &17.3 & 153.5\\
200        & 0.50& $100$  & 15.5 & 14.5& 204.9\\
        &     & $200$   & 15.8 &19.1 &314.0\\
        \cline{2-6}
        &     & $50$    & 11.3 &11.9 & 52.5\\
        & 0.80& $100$  & 14.2 & 15.1& 68.9\\
        &     & $200$  & 15.2 & 17.5& 112.8\\
             \hline
          
        \end{tabular}
       
    \end{minipage} 
\end{table}

 Tables \ref{res1} and \ref{res2} present a summary of the computational results we have obtained. When comparing the results of DGA with those of the PM, we can conclude that the former consistently finds better solutions, with close running times for high-density graphs. The Wilcoxon rank sum tests yield error probabilities of less than 1\% for the hypothesis that the average objective values from DGA are smaller. However, the devolutionary genetic algorithm although, it still finds higher quality solutions seems to under-perform with networks of a low density and requires high computation times that are closer to those of an exact solving approach (BB). This can be explained by the fact that the linear programming formulation we have used does not produce a very tight relaxation for such graph. Thus, We can globally conclude that the proposed hybrid meta-heuristic approach presents a good compromise between a heuristic and an exact approach, although, its use is not indicated for low-density graphs.  
\subsubsection{Sensitivity Analysis}
Generating the initial population of super-optimal solutions is unsurprisingly the phase of the algorithm that requires the biggest computational effort. The size of the initial population and the tightness of the relaxation are the aspect that seem to have the most influence both on the quality of the feasible solutions reached and on running time. Only the former of these two aspects being controllable, the primary adjustment parameter for DGA is thus the size of the initial population of super-optimal solutions. In Figure \ref{size}, we study the influence of the number of initial super-optimal solutions on the quality of generated solutions and on running time for the same set of a $100$ instances. In this graph, the X axis represents the number of initial solutions, while the Y axis represents the average percentage of variation in the value of the best generated solution, compared to the optimum generated by BB, as well as the average percentage of variation in computation time, compared to the computation time when starting with two super-optimal solutions. 

\begin{figure}[!htb]
\begin{center}
\caption{\label{size} Influence of the size of the initial population on value and computation time (averages for $100$ instances)} 
\includegraphics[width=0.75\linewidth]{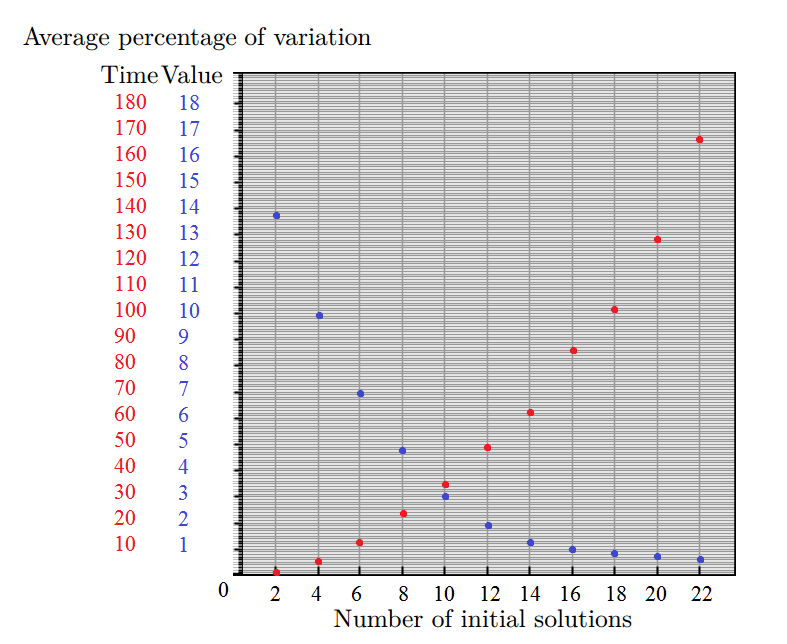}
\end{center}
\end{figure}

We can observe that if this size is too small, then the form of the resulting solutions is relatively restricted from the get-go leading to fast convergence but with a higher probability of convergence to a local optimum. If, on the other hand, the size is too large, then a disproportionate amount of computational time is required. We have observed that increasing the initial population size beyond a dozen, although increasing the computational effort does not significantly improve the quality of the feasible solutions that are eventually generated.

\section{Conclusion}
\label{conclusion}
The preliminary experiments we have performed support the use of devolutionary algorithms for the MLST problem and their development for other NP-hard combinatorial optimization problems. Ongoing investigation will consist in evaluating their results for larger instances with low densities. A comparison of different linear programming formulations of the MLST problem, such as the work done in \cite{15} for the Steiner tree problem, and in \cite{16} for the minimum labeling spanning problem, is outside the scope of the current research. However, such a study would certainly have been profitable to the design of hybrid meta-heuristic approaches like the one we propose in this paper, in addition to its obvious usefulness for computing lower-bounds to the problem in exact solving approaches.


\begin{thebibliography}{xx}
\bibitem{SA} Dougherty, M. J., 1998. Is the Human Race Evolving or Devolving?. Scientific American, July 20 1998. 
\bibitem{STOP}{Safe, M., Carballido, J., Ponzoni, I., Brignol, N., 2004. On Stopping Criteria for Genetic Algorithms. Advances in Artificial Intelligence, SBIA 2004, Lecture Notes in Computer Science, v3171:405-413.} 
\bibitem{STOP2}{Studniarski, M., 2010. Stopping Criteria for Genetic Algorithms with Application to Multiobjective Optimization. Parallel Problem Solving from Nature, PPSN XI, Lecture Notes in Computer Science 6238:697-706.} 
\bibitem{local}{Rocha, M., Neves, J., 1999. Preventing Premature Convergence to Local Optima in Genetic Algorithms via Random Offspring Generation. Multiple Approaches to Intelligent Systems, Lecture Notes in Computer Science, 1611:127-136.} 
\bibitem{111}{Freire, H., Oliveira, P. M., Solteiro Pires, E. J., Bessa, M., 2015. Many-objective optimization with corner-based search. Memetic Computing, 7(2):105-118.}
\bibitem{ML}{Shapiro, J., 2001. Genetic Algorithms in Machine Learning. Machine Learning and Its Applications, Lecture Notes in Computer Science, 2049:146-168.} 



\bibitem{7} Hakimi, S. L., 1971. Steiner's problem in graphs and its implications. Networks, 1: 113-133.
\bibitem{8} Chang, R.S., Leu, S.J, 1997.  The minimum labelling spanning tree. Inf. Process. Lett. 63(5): 277-282.
\bibitem{9} Kapsalis, A., Rayward-Smith, V. J., Smith, G. D., 1993. Solving the Graphical Steiner Tree Problem Using Genetic Algorithms. The Journal of the Operational Research Society, 44(4):397-406.
\bibitem{10} Lai, X., Zhou, Y., He, J., Zhang, J., 2013. Performance Analysis of Evolutionary Algorithms for the Minimum Label Spanning Tree Problem. IEEE Transactions on Evolutionary Computation, 18(6):860-872. 
\bibitem{1} Cerulli, R., Fink, A., Gentili, M., Voss, S., 2006. Extensions of the minimum labelling spanning tree problem. Journal of Telecommunications and Information Technology 4:39-45.
\bibitem{2} Consoli, S., Darby-Dowman, K., Mladenovic, N., Moreno-Perez, J.A., 2009. Variable neighbourhood search for the minimum labelling Steiner tree problem. Annals of Operations Research 172 (1), 71-96.
\bibitem{6} Consoli, S., Moreno-Perez, J.A., Darby-Dowman, K., Mladenovic, N., 2008. Discrete Particle Swarm Optimization for the minimum labelling Steiner tree problem. In: Krasnogor, N., Nicosia, G., Pavone, M., Pelta D. (Eds.), Nature Inspired Cooperative Strategies for Optimization, volume 129 of Studies in Computational Intelligence, pages 313-322, Springer-Verlag, New York.
\bibitem{11} Blum, C., Aguilera, M., Roli, A., Sampels, M., 2008. Hybrid Metaheuristics: An Emerging Approach to Optimization (1st ed.). Springer Publishing Company, Incorporated.
\bibitem{12} Hu, B., Leitner, M., Raidl, G. R., 2008. Combining variable neighborhood search with integer linear programming for the generalized minimum spanning tree problem. Journal of Heuristics, 14(5):473-499.
\bibitem{121} Nekkaa, M., Boughaci, D., 2015. A memetic algorithm with support vector machine for feature selection and classification. Memetic Computing, 7(1):59-73.
\bibitem{122} Acampora, G., Panigrahi, B. K., 2015. Thematic issue on hybrid nature-inspired algorithms: concepts, analysis and applications. Memetic Computing, 7(1):1-2.
\bibitem{123} Barril Otero, F. E., Masegosa, A. D., Terrazas, G., 2014. Thematic issue on advances in nature inspired cooperative strategies for optimization. Memetic Computing, 6(3):147-148.
\bibitem{124} Tawhid, M. A., Fouad, A., 2016. A simplex social spider algorithm for solving integer programming and minimax problems. Memetic Computing, First Online 16 February 2016:1-20.
\bibitem{125} Alba, E., Resende, M. G. C., Urquhart, M. E., Lim, M.-H., 2012. Thematic Issue on Memetic Algoriths: Theory and applications in OR/MS. Memetic Computing, 4(2):87-88.
\bibitem{17} Beasley, J. E., 1989. An SST-based algorithm for the Steiner problem in graphs, Networks 19:1-16.
\bibitem{15} Polzin, T., Daneshmand, S. V., 2001. A comparison of Steiner tree relaxations. Discrete Applied Mathematics, 112:241-261.
\bibitem{ABACUS} Junger, M., Thienel, S., 2001. The ABACUS System for Branch-and-Cut-and-Price Algorithms in Integer Programming and Combinatorial Optimization. Software: Practice and Experience, 30:1325-1352.
\bibitem{16} Chwatal, A. M., Raidl, G. R., 2011. Solving the Minimum Label Spanning Tree Problem by Mathematical Programming Techniques. Advances in Operations Research, vol. 2011, Article ID 143732.



\end{thebibliography}
\end{document}